\documentclass[12pt,a4paper]{amsart}

\usepackage{amsmath,amsfonts,amssymb,amsthm}
\usepackage{mathtools}

\theoremstyle{plain}

\newtheorem{thm}{Theorem}[section]

\newtheorem{lem}[thm]{Lemma}
\newtheorem{prop}[thm]{Proposition}

\newtheorem{thmABC}{Theorem}

\newtheorem{corABC}[thmABC]{Corollary}

\theoremstyle{definition}

\newtheorem*{ackn}{Acknowledgements}

\numberwithin{equation}{section}

\newcommand{\ol}[1]{\overline{#1}}

\renewcommand{\d}{\mathrm{d}}

\title[Lower rank of direct products]{The lower rank of direct
  products of hereditarily just infinite groups}

\author{Benjamin Klopsch} \address{Mathematisches Institut der
  Heinrich-Heine-Universit\"at, Universit\"atsstr.\ 1, 40225
  D\"usseldorf, Germany}\email{klopsch@math.uni-duesseldorf.de}

\author{Matteo Vannacci} \address{Mathematisches Institut der
  Heinrich-Heine-Universit\"at, Universit\"atsstr.\ 1, 40225
  D\"usseldorf, Germany}\email{vannacci@math.uni-duesseldorf.de}



\begin{document} 

\maketitle


\begin{abstract}
  We determine the lower rank of the direct product of finitely many
  hereditarily just infinite profinite groups of finite lower rank.
\end{abstract}


\section{Introduction} 

For primes~$p$, the theory of $p$-adic analytic pro-$p$ groups plays a
central part in the study of pro-$p$ groups and has interesting
applications in infinite group theory; see~\cite{MR0209286,MR1720368}.
According to a well-known algebraic characterisation, a pro-$p$ group
$G$ is $p$-adic analytic if and only if $G$ has an open subgroup of
finite rank.  Here the \emph{rank} of $G$ is defined as
\[
\mathrm{rk}(G) = \sup \{ \d(H) \mid H \le_\mathrm{o} G \},
\] 
where $\d(H)$ denotes the minimal number of topological generators of
the open subgroup $H \le_\mathrm{o} G$.  Indeed, Lubotzky and
Mann~\cite{lubotzky:powerfulpgroups} even established the following
refinement: a pro-$p$ group $G$ is $p$-adic analytic if and only if
 the \emph{upper rank}
$\overline{L}_{\d}(G) = \limsup \{\d(H) \mid H \le_\mathrm{o} G\}$
is finite, where the limit superior is taken over the net of open
subgroups ordered by reverse inclusion; moreover, in this case
$\overline{L}_{\d}(G)$ is equal to the dimension $\dim(G)$ of $G$ as a
$p$-adic manifold. 

Lubotzky and Mann also introduced, for a general
profinite group~$G$ the \emph{lower rank}
\[
\mathrm{lr}(G) = \underline{L}_{\, \d}(G) = \liminf \{\d(H)\ \vert\
H\le_\mathrm{o} G\},
\]
where again the limit inferior is taken over the net of
open subgroups ordered by reverse inclusion. They proved that the
lower rank of a compact $p$-adic analytic group coincides with the
number of generators of its associated $\mathbb{Q}_p$-Lie algebra.  By
a classical theorem of Kuranishi~\cite{MR0041145}, this implies that
the lower rank of any compact $p$-adic analytic group with semi-simple
$\mathbb{Q}_p$-Lie algebra is equal to~$2$.

Lubotzky and Shalev~\cite{MR1264349} continued the study of the lower
rank and showed, for instance, that there exist non-analytic pro-$p$
groups of finite lower rank.  Refining their techniques,
Barnea~\cite{barnea:lowerrank} established, for instance, that the
lower rank of the $\mathbb{F}_p[\![t]\!]$-analytic group
$\mathrm{SL}_2(\mathbb{F}_p[[t]])$ is equal to~$2$.  Computing the
lower ranks of profinite groups is usually rather challenging and
producing new families of groups of finite lower rank is of
considerable interest.

In this paper we are interested in the lower ranks of finite direct
products of hereditarily just infinite profinite groups.  We recall
that a profinite group $G$ is just infinite, if $G$ is
infinite and every non-trivial closed normal subgroup
$N \trianglelefteq_\mathrm{c} G$ is open in~$G$.  The group $G$ is
\emph{hereditarily just infinite} if every open subgroup
$H \le_\mathrm{o} G$ is just infinite.

There are many non-(virtually abelian) hereditarily just infinite
profinite groups of finite lower rank, including those that are
$p$-adic analytic; cf.\ \cite{MR1483894}.  However, Ershov and
Jaikin-Zapirain proved that there exist also hereditarily just
infinite pro-$p$ groups of infinite lower rank; see
\cite[Corollary~8.10]{ershov:weighteddeficiency}.  More recently, the
second author of the present paper analysed the lower rank in a family
of non-(virtually pro-$p$) hereditarily just infinite groups and
conjectured that there exist such groups of any given lower rank
in~$\mathbb{N}_{\ge 2} \cup \{ \infty \}$; see~\cite{MR3466595}.  We
establish the following theorem and corollary.

\begin{thmABC} \label{thm:directprodhji} Let $G = \prod_{i=1}^n G_i$
  be a non-trivial direct product of finitely many hereditarily just infinite
  profinite groups of finite lower rank.  Set
    \[
    d = \max \{\mathrm{lr}(G_i) \mid 1 \le i \le n \} \quad \text{and}
    \quad r = \max \{ r_p \mid p \text{ prime} \},
    \]
    where $r_p$ denotes the number of indices $i$
    such that $G_i$ is virtually-$\mathbb{Z}_p$.  Then the lower rank
    of $G$ is $\mathrm{lr}(G) = \max \{d,r\}$.
\end{thmABC}

\begin{corABC}
  The non-trivial direct product $G = \prod_{i=1}^n G_i$ of finitely many pairwise
  non-commensurable hereditarily just infinite profinite groups of
  finite lower rank has lower
  rank $$\mathrm{lr}(G) = \max \{\mathrm{lr}(G_i) \mid 1 \le i \le n\}.$$
\end{corABC}

In the proof we use basic facts about the structure of just infinite
profinite groups, in particular a result of Reid~\cite{MR2684140}.  
We emphasise that, in general, the lower rank of a direct product of
profinite groups can be as large as the sum of the lower ranks of the
factors. For instance, the lower rank of a free abelian pro-$p$ group
$\mathbb{Z}_p \times \ldots \times \mathbb{Z}_p$, with $n$ factors, is
clearly~$n$.

Theorem~\ref{thm:directprodhji} can be regarded as a generalisation of
the aforementioned result, due to Kuranishi, Lubotzky and Mann, that
every compact $p$\nobreakdash-adic analytic group with semi-simple
$\mathbb{Q}_p$-Lie algebra has lower rank~$2$.  Finally, we remark
that Theorem~\ref{thm:directprodhji} can be applied to the family of
hereditarily just infinite profinite groups of finite lower rank
described in~\cite{MR3466595}.  In this way we obtain many new
examples of profinite groups of finite lower rank.


\section{Preliminaries}\label{sec:directprodhji}

Clearly, profinite groups of finite lower rank are finitely generated
and thus countably based.  Restricting to the latter class of
profinite groups, we can navigate around the general notion of the
limit inferior of a net.  The \emph{lower rank} of a
countably based profinite group $G$ is
\begin{multline*}
  \mathrm{lr}(G) = \min \Big\{ \sup \big\{ \inf\{\d(H_i)\mid i\ge N\}
  \mid N\in \mathbb{N} \big\} \mid (H_i)_{i\in \mathbb{N}} \in
  \mathcal{C}(G)
  \Big\} \\
  \in \mathbb{N}_0 \cup \{ \infty \},
\end{multline*}
where
\[
\mathcal{C}(G) = \Big\{ (H_i)_{i\in \mathbb{N}} \mid G=H_1
\prescript{}{\mathrm{o}}\ge\, H_2 \prescript{}{\mathrm{o}}\ge\, \ldots
\text{ and } \bigcap\nolimits_{i\in \mathbb{N}} H_i = 1 \Big\}
\]
is the collection of all descending chains of open subgroups of $G$ that form
a neighbourhood base for the identity element. 

In other words, a countably based profinite group $G$ has lower rank at
most $r$ if there exists a descending chain of $r$-generated open
subgroups of $G$ that form a neighbourhood base for the identity.

\smallskip

In preparation for the proof of Theorem~\ref{thm:directprodhji} we
collect two basic lemmata.  Recall that a profinite group $G$
possesses \emph{virtually} a group-theoretic property $\mathfrak{P}$
if $G$ has an open subgroup $H$ that has $\mathfrak{P}$. 
  We abbreviate ``virtually-(infinite procyclic pro-$p$)'' to
  ``virtually-$\mathbb{Z}_p$''.
 
\begin{lem}\label{lem:virtabhji}
  Let $G$ be a virtually abelian, hereditarily just infinite profinite
  group. Then $G$ is virtually-$\mathbb{Z}_p$ for a suitable prime~$p$.
\end{lem}

\begin{proof}
  Let $A$ be an open abelian subgroup of~$G$.  As $A$ is just
  infinite, it is a pro-$p$ group for some prime~$p$ and infinite
  pro-cyclic.  Thus $A$ is isomorphic to $\mathbb{Z}_p$.
\end{proof}

\begin{lem}\label{lem:virtuallyabelian}
  Let $H$ be a just infinite profinite group that is not virtually
  abelian, and let $L \le_\mathrm{o} H$. Then there exists $x \in H$
  such that $L$ is not contained in $\mathrm{C}_H(x)$.
\end{lem}

\begin{proof}
  For a contradiction, assume that $L \subseteq \mathrm{C}_H(x)$ for
  all $x \in H$.  Then $L$ is contained in the centre~$\mathrm{Z}(H) =
  1$, hence $L = 1$. 
\end{proof}


\section{Proof of Theorem~\ref{thm:directprodhji}}

Before proving Theorem~\ref{thm:directprodhji} we establish another
auxiliary result.

\begin{lem}\label{lem:noniso}
  Let $n\in \mathbb{N}$ and let $G = \prod_{i=1}^n G_i$ be a direct product of finitely many
  hereditarily just infinite profinite groups of finite lower rank,
  where none of them is virtually abelian, and define the integer
  $d = \max \{\mathrm{lr}(G_i) \mid 1 \le i \le n \}$. Then for every
  basic open neighbourhood $\prod_{i=1}^n U_i$ of the identity element
  in $G$, with $U_i \subseteq_\mathrm{o} G_i$, there exist open
  subgroups $H_i \le_\mathrm{o} G_i$, for $1 \le i \le n$, such that
  \begin{enumerate}
  \item $H_i \subseteq U_i$ and $\d(H_i) \le d$ for $1\le i\le n$,
  \item $H_i \not\cong H_j$ for $1\le i < j \le n$.
  \end{enumerate}
\end{lem}

\begin{proof}
  For each $i \in \{1, \ldots, n\}$, the group $G_i$ admits a
  descending chain of open $d$-generated subgroups
  $H_{i,1} \gneqq H_{i,2} \gneqq \ldots$ satisfying $H_{i,k} \subseteq U_i$
  for $k \in \mathbb{N}$.  By \cite[Theorem~E]{MR2684140}, a
  non-(virtually abelian) just infinite profinite group does not
  contain any proper open subgroups isomorphic to the whole
  group. Hence, for $1 \leq i \leq n$, the groups $H_{i,k}$,
  $k \in \mathbb{N}$, are pairwise non-isomorphic.  Consequently there
  are $k_1, \ldots, k_n \in \mathbb{N}$ such that
  $H_1 = H_{1,k_1}, \ldots, H_n = H_{n,k_n}$ are pairwise
  non-isomorphic.
\end{proof}

Now, let $G = \prod_{i=1}^n G_i$ be a direct product of finitely many
hereditarily just infinite profinite groups of finite lower rank.  
We set
\[
\ell = \max \{ d,r \},
\]
where $d = \max \{\mathrm{lr}(G_i) \mid 1 \le i \le n \}$ and
$r = \max \{ r_p \mid p \text{ prime} \}$ are defined as in the
statement of Theorem~\ref{thm:directprodhji}; here $r_p$ denotes the
number of $i \in \{1, \ldots, n\}$ such that $G_i$ is
virtually-$\mathbb{Z}_p$.  

Clearly $\mathrm{lr}(G)\ge \ell$.
Let $U = \prod_{i=1}^n U_i$ be a basic open neighbourhood of the
identity element in $G$, with $U_i \subseteq_\mathrm{o} G_i$.  We need
to find an $\ell$-generated open subgroup $K \le_\mathrm{o} G$ with
$K \subseteq U$.

Without loss of generality, the first $m$ factors $G_1$, \ldots, $G_m$
are not virtually abelian, while the remaining $n-m$ factors
$G_{m+1}$, \ldots, $G_n$ are virtually abelian.  By
Lemma~\ref{lem:virtabhji}, there exists, for each
$i \in \{m+1, \ldots, n\}$, a prime $q_i$ such that $G_i$ is
virtually-$\mathbb{Z}_{q_i}$.  Reordering the factors and descending
to an appropriate open subgroup $C \le_\mathrm{o} \prod_{i=m+1}^n G_i$
we can arrange that $C \subseteq \prod_{i=m+1}^n U_i$ and
\[
C = C_1 \times \ldots \times C_s, \quad \text{with } C_i =
\overline{\langle y_{i1}, \ldots, y_{i  t_i} \rangle}
\cong \mathbb{Z}_{p_i}^{\, t_i} \text{ for $1 \le i \le s$},
\]
where $s \in \mathbb{N}\cup \{0\}$ with $s \le n-m$, the positive
integers $t_i = r_{p_i}$
satisfy $\sum_{i=1}^s t_i = n-m$ and $p_1, \ldots, p_s$ denote
distinct primes.  It is convenient to set $y_{ij} = 1$ for
$1 \leq i \leq s$ and $t_i +1 \le j \le \ell$ as well as for
$s+1 \le i \le n$ and $1 \leq j \leq \ell$. We now work in the open
subgroup $\prod_{i=1}^m G_i \times C \le_\mathrm{o} G$.

By Lemma~\ref{lem:noniso}, we can choose subgroups
$H_i \le_\mathrm{o} G_i$ with generators $h_{i1},\ldots,h_{id}$, for
$1 \le i \le m$, such that
\[
H_i = \ol{\langle h_{i1},\ldots,h_{id} \rangle} \subseteq U_i \qquad
\text{and} \qquad H_i \not\cong H_j \quad \text{for
  $1 \leq i < j \le m$}.
\]
Again it is convenient to set $h_{ij} =1$ for $1 \le i \le m$ and
$d+1 \le j \le \ell$ as well as for $m+1 \le i \le n$ and
$1 \le j \le \ell$.  We write $H=\prod_{i=1}^m H_i$ for the internal
direct product of $H_1,\ldots,H_m$.  

To conclude the proof, it suffices to produce a $\ell$-generated open
subgroup $K \le_\mathrm{o} H \times C \le_\mathrm{o} G$.  We consider
\[
K = \ol{\langle g_1, g_2, \ldots, g_ \ell \rangle}
\le_\mathrm{c} H_1 \times \cdots \times H_m \times C = H \times C,
\]
where $g_i = h_i y_i$ with 
\[
h_i = h_{1i} h_{2i} \cdots h_{ni} \in H \quad \text{and} \quad y_i = y_{1i}
y_{2i} \cdots y_{ni} \in C \quad \text{for $1 \le i \le \ell$.}
\]
Clearly, $K$ is $\ell$-generated.  Furthermore, $K$ is a sub-direct
product of $H_1$, \ldots, $H_m$ and $C$, i.e., it is a closed subgroup
of $H_1 \times \ldots \times H_m \times C$ that projects onto each of
the $m+1$ direct factors.  The proof of
Theorem~\ref{thm:directprodhji} can therefore be completed by
appealing to the next proposition, which is also of independent
interest.

\begin{prop}\label{prop:subdirect}
  Let $m \in \mathbb{N}_0$, and let
  \[
  K \le_\mathrm{c} H_1 \times \ldots \times H_m \times
  C
  \]
  be a sub-direct product of $m$ pairwise non-isomorphic finitely
  generated just infinite profinite groups $H_1, \ldots, H_m$ that are
  not virtually abelian and a finitely generated abelian profinite
  group~$C$. 

  Then $K$ is an open subgroup of $\prod_{i=1}^m H_i \times C$.
\end{prop}

\begin{proof}
  We may assume that $m\ge 1$.  Put $H = \prod_{i=1}^m H_i$. For each
  $j \in \{1,\ldots,m\}$, let $\pi_j \colon H \times C \to H_j$ denote
  the projection onto $H_j$; and let $\pi_C \colon H \times C \to C$
  denote the projection onto~$C$.  Fix finitely many generators
  $h_1,\ldots,h_d$ for $H_1$ and $y_1,\ldots,y_s$ for $C$ so that
  \[
  H_1=\overline{\langle h_1,\ldots,h_d \rangle} \quad \text{and} \quad
  C=\overline{\langle y_1,\ldots,y_s\rangle}.
  \]
  Since $K$ is a sub-direct product, we find
  $h_1^\ast,\ldots,h_d^\ast, y_1^\ast,\ldots, y_d^\ast \in K$ such
  that $h_i^\ast \pi_1 = h_i$ for $1 \le i \le d$ and
  $y_j^\ast \pi_C = y_j$ for $1 \le j \le s$.

  We observe that it suffices to show that, for each
  $j \in \{1,\ldots,m\}$, there exist an open subgroup
  $K_j \le_\mathrm{o} H_j$ with $K_j \le K$.  For then we get
  $\widetilde{K}= K_1 \times \cdots \times K_m \le_\mathrm{o} H$ and,
  setting $N = \lvert H : \widetilde{K} \rvert !$, we obtain
  $h^{N} \in \bigcap \{ \widetilde{K}^g \mid g \in H \} \subseteq
  \widetilde{K}$
  for every $h \in H$.  As $C$ is central, this implies
  $y_i^{N} = (y_i^\ast)^N ((y_i^\ast)^{-1} y_i)^N \in K$ for
  $1\le i\le s$, and we deduce from
  $\widetilde{K} \times \ol{\langle y_1^N,\ldots,y_s^N \rangle}
  \le_\mathrm{o} H \times C$ that $K \le_\mathrm{o} H\times C$.

  It remains to construct the aforementioned subgroups
  $K_j \le_\mathrm{o} H_j$ with $K_j \le K$ for $1 \le j \le m$.  By
  symmetry, it is enough to manufacture~$K_1$.  Indeed, we construct
  recursively, for each $1 \le i \le m$, a subgroup
  $K^{(i)} \le_\mathrm{c} K$ such that
  \[
  K^{(i)} \leq H_1 \times H_{i+1} \times H_{i+2} \times \cdots \times
  H_m \qquad \text{and} \qquad K^{(i)} \pi_1 \le_\mathrm{o} H_1.
  \]
  Then we take $K^{(m)}$ for $K_1$ and the proof is complete.

  Note that $K^{(1)} = [K,K]$ satisfies the relevant conditions,
  because $H_1 = K \pi_1$ is just infinite and non-abelian.  Now
  suppose that for $i \in \{2,\ldots,m\}$ the group $K^{(i-1)}$ is
  already available and build $K^{(i)}$ as follows.  Let $F = F_d$
  denote the free profinite group on $d$ generators $a_1,\ldots,a_d$
  and define profinite presentations
  \[
  1 \to R_i \rightarrow F \xrightarrow{\varphi_i} H_i \to 1, \qquad 1
  \le i \le m,
  \]
  with $a_j \varphi_1 = h_{j}$ for $1 \le j \le d$.

  Recall that $F/R_1\cong H_1$ and $F/R_i\cong H_i$ are non-isomorphic
  just infinite groups.  This gives $R_i \not \subseteq R_1$, and
  $R_iR_1/R_1$ is a non-trivial closed normal subgroup of
  $F/R_1 \cong H_1$.  This implies
  $R_i \varphi_1 \trianglelefteq_\mathrm{o} H_1$, and we obtain
  $R_i \varphi_1 \cap K^{(i-1)} \pi_1 \le_\mathrm{o} K^{(i-1)} \pi_1$.
  By Lemma~\ref{lem:virtuallyabelian}, there exists
  $x_{i-1} \in K^{(i-1)}$ such that
  $R_i \varphi_1 \cap K^{(i-1)} \pi_1 \not \le \mathrm{C}_{K^{(i-1)}
    \pi_1}(x_{i-1} \pi_1)$.

  Consequently, we find a word $w_i \in R_i$ such that
  $w_i(h_{1},\ldots,h_{d}) \notin \mathrm{C}_{K^{(i-1)} \pi_1}(x_{i-1}
  \pi_1)$.
  Using the properties of $K^{(i-1)}$ and
  $w_i \in \mathrm{ker}(\pi_i)$, it follows that
  $z_i = [w_i(h_1^\ast,\ldots,h_d^\ast),x_{i-1}] \in K$ satisfies
  \begin{equation}\label{eq:gamma}
    z_i \equiv b \pmod{H_{i+1} \times \cdots \times H_m}, \quad \text{where $1
      \ne b \in H_1$.}
  \end{equation}
  Set $K^{(i)} = \ol{\langle z_i \rangle^{K}} \le_\mathrm{c} K$.
  Visibly $z_i \in H_1 \times H_{i+1} \times \cdots \times H_m$, hence
  $K^{(i)} \leq H_1 \times H_{i+1} \times \cdots \times H_m$.
  Moreover, $K^{(i)} \pi_1 = \ol{\langle b \rangle^{H_1}}$ is a
  non-trivial closed normal subgroup of the just infinite group $H_1$,
  so $K^{(i)} \pi_1 \le_\mathrm{o} H_1$.
\end{proof}

\begin{ackn}
  A weak version of Theorem~\ref{thm:directprodhji} formed part of the
  second author's PhD thesis, Royal Holloway, University of London,
  2015. We thank the referee for encouraging us to state
  Proposition~\ref{prop:subdirect} as a separate result.
\end{ackn}



\end{document}